\documentclass[10pt]{amsart}
\usepackage[margin=1in,letterpaper,portrait]{geometry}
\usepackage{amsmath,amssymb,paralist, comment}
\usepackage{latexsym,graphicx,hyperref}
\usepackage[dvipsnames]{xcolor}
\usepackage{tikz}
\usetikzlibrary{arrows,shapes,trees,automata}
\usepackage[normalem]{ulem}

\usepackage{caption}

\usepackage{varwidth}
\DeclareCaptionFormat{myformat}{%
  \begin{varwidth}{\linewidth}%
    \centering
    #1#2#3%
  \end{varwidth}%
}

\usepackage{float}

\theoremstyle{plain}
   \newtheorem{theorem}{Theorem}[section]
   \newtheorem{proposition}[theorem]{Proposition}
   \newtheorem{prop}[theorem]{Proposition}
   
   \newtheorem{corollary}[theorem]{Corollary}
    
   \newtheorem{conjecture}[theorem]{Conjecture}
   
   \newtheorem*{theorem*}{Theorem}
\theoremstyle{definition}
   \newtheorem{definition}[theorem]{Definition}
   \newtheorem{example}[theorem]{Example}
   
   \newtheorem{question}[theorem]{Question}
   \newtheorem{remark}[theorem]{Remark}

\numberwithin{equation}{section}

\newcommand\Symm{\mathfrak{S}}

\newcommand{\defeq}{\overset{\text{def}}{=\hspace{-4pt}=}}

\newcommand{\id}{\mathrm{id}}

\newtheoremstyle{TheoremNum}
        {}{}                            
        {\itshape}                      
        {}                              
        {\bfseries}                     
        {.}                             
        { }                             
        {\thmname{#1}\thmnote{ \bfseries #3}}
    \theoremstyle{TheoremNum}

\makeatletter
\newcommand*{\rom}[1]{\expandafter\@slowromancap\romannumeral #1@}
\makeatother
 
\begin{document}

\title[Fully commutative elements in complex reflection groups]{A note on fully commutative elements in \\complex reflection groups}
\author{Jiayuan Wang}
\address{Department of Mathematics \\ George Washington University \\ Washington, DC, USA 20052}
\email{j453w588@gwu.edu}

\maketitle

\begin{abstract}
Fully commutative elements in types $B$ and $D$ are completely characterized and counted by Stembridge. Recently, Feinberg-Kim-Lee-Oh have extended the study of fully commutative elements from Coxeter groups to the complex setting, giving an enumeration of such elements in $G(m,1,n)$. In this note, we prove a connection between fully commutative elements in $B_n$ and in $G(m,1,n)$, which allows us to characterize fully commutative elements in $G(m,1,n )$ by pattern avoidance. Further, we present a counting formula for such elements in $G(m,1,n)$. 

\end{abstract}

\section{Introduction}
\label{sec:introduction}
Let $G$ be the group generated by a set of elements $\{s_1, s_2, \ldots\}$. Assume $g\in G$ has a reduced expression $g=s_{i_1}s_{i_2}\cdots s_{i_{\ell(g)}}$ where \emph{length} $\ell(g)$ is the minimum number of generators needed. If any other reduced expression of $g$ can be obtained from $s_{i_1}s_{i_2}\cdots s_{i_{\ell(g)}}$ by interchanges of adjacent commuting generators, i.e., commutation relations, then $g$ is \emph{fully commutative}. 

Consider the symmetric group $\Symm_n$ generated by simple transpositions. Fully commutative elements are permutations that avoid the pattern $321$ \cite{BJS}. The number of such permutations is the $n$-th Catalan number. When $G$ is a Weyl group, fully commutative elements can be characterized by root system \cite{BP05} and Lusztig's $a$-function \cite{BF98, Shi05b}. When $G$ is a simply laced Coxeter group, fully commutative elements can also be described by root system \cite{FS97}. When $G$ is the Coxeter group $B_n$ or $D_n$, fully commutative elements can be viewed as the linear extensions of a heap \cite{S96} and then be characterized by pattern avoidance \cite{S97}. We will discuss Stembridge's work on pattern avoidance more extensively in Section \ref{sec:background}.

Recently, \cite{FKLO} generalized the study of fully commutative elements from Coxeter groups to the complex setting, proving a counting formula for these elements that agrees with \cite{S96} in $B_n$. The purpose of this note is to further explore the connection between fully commutative elements in $B_n$ and in $G(m,1,n)$. We will show that an element $g$ in $G(m,1,n)$ is fully commutative if and only if the resulting element, after every nontrivial $m$-th root of unity in $g$ is replaced with $-1$, is fully commutative in $B_n$.\footnote{After posting this note to arXiv, we were made aware that an equivalent result was independently obtained by Thomas Magnuson in his undergraduate honors thesis \cite{Mag}. We thank Tianyuan Xu for bringing this to our attention.} This implies that the pattern avoidance of fully commutative elements in $B_n$ extends to $G(m,1,n)$. As a consequence, we will prove that the enumeration of fully commutative elements in $G(m,1,n)$ can be achieved by counting fully commutative elements in $B_n$ by the number of $-1$'s in their matrix forms. Further, we will discuss why fully commutative elements in $G(m,m,n)$ do not have the pattern avoidance property.

The remainder of this note is as follows: in Section \ref{sec:background}, we give relevant background on Coxeter groups and complex reflection groups, and introduce Stembridge's work in $B_n$ and the result by Feinberg-Kim-Lee-Oh in $G(m,1,n)$. In Section \ref{sec:main}, we prove our main theorem, on the connection between fully commutative elements in $B_n$ and in $G(m,1,n)$.  Then in Section \ref{sec:open}, we consider fully commutative elements in groups $G(m,m,n)$ with different generating sets and Shephard groups. We present some preliminary data and propose a few conjectures and questions for future directions. 
\subsection*{Acknowledgements}
The author thanks Kyu-Hwan Lee for giving a talk in Sage Days FPSAC 2019 that introduced the idea and is deeply grateful to Joel Brewster Lewis for many insightful discussions and useful comments at all stages of this project, and Alejandro Morales and Theo Douvropoulos for helpful conversations.

\section{Background}
\label{sec:background}

\subsection{Coxeter groups}
For a thorough treatment on Coxeter groups, see \cite{BB,JH}. Let $W$ be a group with a set of generators $S\subseteq W$, subject only to relations of the form 
\[ (ss')^{m(s,s')} =1,\]
where $m(s,s)=1$, $m(s,s')=m(s',s)\geq 2$ for $s\neq s'$ in $S$, with the convention that $m(s, s')=\infty$ when no relation occurs for a pair $s, s'$. Then $W$ is a \emph{Coxeter group}. 

One key example is the symmetric group $\Symm_n$, which can be realized as permutations. The hyperoctahedral group $B_n$ of all \emph{signed permutations} of the set $\{1, 2, \ldots, n\}$ (permutations of the set $\{-n, \ldots, -1, 1, \ldots, n\}$)  is another example. Every element in $B_n$ is a \emph{monomial matrix}, i.e., in each row and column, there is exactly one nonzero entry, whose nonzero entries are either $1$ or $-1$. The group $B_n$ has a subgroup $D_n$, whose elements are monomial matrices that have an even number of $-1$'s. Thus, $D_n$ is called the group of \emph{even-signed permutations}, which is also a Coxeter group. In addition, the dihedral group $I_2(m)$ is a Coxeter group. 

One may notice that these examples are all \emph{finite real reflection groups}: a finite group generated by \emph{real reflections}, which are linear transformations that fix a hyperplane in a (real) Euclidean space. In fact, finite real reflection groups are the finite Coxeter groups. 

\subsection{Complex reflection groups}
For a general reference on complex reflection groups, see \cite{LehrerTaylor}. Given a finite-dimensional complex vector space $V$, a \emph{reflection} is a linear transformation $t: V \rightarrow V$ whose fixed space $\operatorname{ker}(t-1)$ is a hyperplane (has codimension $1$), and a finite subgroup $G$ of $GL(V)$ is called a \emph{complex reflection group} if $G$ is generated by its subset of reflections. Complex reflection groups were classified by Shephard and Todd~\cite{ST54}: every complex reflection group is a direct product of irreducibles, and every irreducible is isomorphic to a group of the form
\[
G(m, p, n) \defeq \left\{\begin{array}{l} n \times n \text{ monomial matrices whose nonzero entries are}\\m\text{th} \text{ roots of unity with product a } \frac{m}{p}\text{th} \text{ root of unity}\end{array} \right\}
\]
for positive integers $m$, $p$, $n$ with $p \mid m$ or to one of $34$ exceptional examples, denoted as $G_4$, \ldots, $G_37$.

For every $m$, $p$, $n$, there is a projection map
\[ f: G(m,p,n) \twoheadrightarrow G(1,1,n) = \Symm_n\] 
where $f(g)$ is the result of replacing every root of unity in the matrix of $g$ with $1$. The resulting permutation $f(g)$ is the \emph{underlying permutation} of $g$. We may use a shorthand to refer to an element $g\in G(m,p, n)$ as $[f(g); (a_1, \ldots, a_n)]$, where $f(g)$ is the underlying permutation of $g$ and $a_i$ is the exponent of the nonzero entry in the $i$-th column. We call $a_i$ the \emph{weight} of the entry and $a_1 +\ldots +a_n \pmod{m}$ the \emph{weight} of the element. For example, in $G(30, 1, 6)$, we have
\begin{equation*}
g=\left[\begin{smallmatrix}
&\omega^{17} &&&&\\
1&&&&&\\
&&&&\omega^{2}&\\
&&\omega^{2}&&&\\
&&&\omega^{3}&&\\
&&&&&\omega^{6}
\end{smallmatrix}\right]
= [214536 ; (0, 17, 2, 3, 2, 6)] = 
[(12)(345)(6) ; (0, 17, 2, 3, 2, 6)],
\end{equation*}
where $\omega=\exp(\frac{2\pi i}{m})$ denotes a fixed primitive $m$th root of unity.

Let $s_1, s_2, \ldots, s_{n-1}$ be simple transpositions, i.e., $s_j = [(j\, j+1); (0,\ldots, 0)]$ and $s_0 = [\id; (1, 0, \ldots, 0)]$ be a diagonal reflection.
When $p =1$, the group $G(m, 1, n)$ can be generated by reflections $s_0, s_1, \ldots, s_{n-1}$. The simple transpositions $s_j$'s have order $2$ and the diagonal reflection $s_0$ has order $n$. 

One recovers the infinite families of real reflection groups as the following special cases: the group $G(1, 1, n)$ is the symmetric group $\Symm_n$; $G(2, 1, n)$ is the signed permutation group $B_n$; $G(2, 2, n)$ is the even-signed permutation group $D_n$; and $G(m, m, 2)$ is the dihedral group $I_2(m)$.

\subsection{Fully commutative elements in Coxeter group $B_n$}
Let $s_1, s_2, \ldots, s_{n-1}$ be simple transpositions, i.e., $s_j = [(j\, j+1); (0,\ldots, 0)]$ and $s_0 = [\id; (1, 0, \ldots, 0)]$ be a diagonal reflection. Then Coxeter group $B_n$ can be generated by $s_0, s_1, \ldots, s_{n-1}$ with defining relations:
\begin{center}
\begin{tabular}{rll}
$(s_0s_1)^4= s_i^2$ & $=1$ & for $1\leq i\leq n-1$,\\
$s_i s_j$ & $= s_j s_i $ & for $i+1 < j\leq n-1$, \\
$s_{i+1}s_i s_{i+1}$ & $ = s_i s_{i+1}s_i$ & for $1\leq i \leq n-2$.\\
\end{tabular}
\end{center}

The shortest left coset representatives for $B_n/B_{n-1}$ are 
\[\{1,\quad s_{n-1}, \quad s_{n-2}s_{n-1}, \quad \ldots,\quad s_0 s_1\cdots s_{n-1}, \quad s_1s_0s_1\cdots s_{n-1}, \quad \ldots, \quad s_{n-1}\ldots s_1s_0 s_1\ldots s_{n-1}\}\]
where shortest means minimal. For $0\leq i\leq j$, denote $[i,j]=s_i\cdot s_{i+1}\cdots s_j$ and $[-i, j]=s_i\cdot s_{i-1}\cdots s_1\cdot s_0\cdot s_1\cdots s_j$. Then rewrite the coset representatives as 
\[\{1,\quad [n-1,n-1] \quad [n-2,n-1], \quad \ldots,\quad [0,n-1], \quad [-1, n-1], \quad \ldots, \quad [-(n-1),n-1]\}.\] 

Using these coset representatives, Stembridge showed in \cite{S97} that every element in $B_n$ has a canonical reduced word
\[
[m_1, n_1]\cdot [m_2, n_2]\cdots [m_r, n_r],
\]
where $n>n_1>\cdots>n_r\geq 0$ and $|m_i|\leq n_i$. 

Further, he proved several equivalent statements regarding full commutativity in $B_n$. In particular, he characterized fully commutative elements by pattern avoidance. Let $g\in B_n$. If g avoids the pattern $(-1,-2)$, it means that $g$ does not contain $\left[\begin{smallmatrix} -1&\\ &-1\end{smallmatrix}\right]$ as a submatrix. Similarly, if $g$ avoids the pattern $(2, 1, -3)$, it means that $g$ does not contain $\left[\begin{smallmatrix} &1&\\ 1&&\\ &&-1\end{smallmatrix}\right]$ as a submatrix. 
\begin{theorem}[{\cite[Cor.~5.6]{S97}}]\label{thrm: type B}
For $g\in B_{n}$, the following are equivalent.
\begin{enumerate}
\item $g$ is fully commutative.
\item In the canonical reduced word $[m_1, n_1]\cdots [m_r, n_r]$ for $g$, we have either
\begin{enumerate}
\item $m_1>\cdots>m_s >m_{s+1}=\cdots=m_r = 0$ for some $s\leq r$, or
\item $m_1> \cdots > m_{r-1}> -m_r >0$.
\end{enumerate}
\item $g$ avoids the pattern $(-1, -2)$ and all patterns $(a, b,c)$ such that $|a|>b>c$ or $-b>|a|>c$.
\end{enumerate}
\end{theorem}

For a complete list of patterns in case $(3)$ in Theorem \ref{thrm: type B}, see Table \ref{patterns}. 
 
\begin{example}
Consider $B_4$ with generating set $\{s_0, s_1, s_2, s_3\}$ where $s_0 = [\id; (1, 0, 0, 0)]$, $s_1 = [(12);(0,\ldots, 0)]$, $s_2 = [(23); (0, \ldots, 0)]$, and $s_3 = [(34);(0, \ldots, 0)]$. Let $g_1 = \left[\begin{smallmatrix} 1&&&\\ &&&1\\&&-1&\\ &1&&\end{smallmatrix}\right]=[(24);(0,0,1,0)]=[2,3]\cdot [-1, 2]=(s_2 s_3) \cdot ( s_1 s_0 s_1 s_2)$. Here are all its reduced expressions:
 \[s_2 {\color{red} s_3} s_1 s_0 s_1 s_2, \,\quad
  s_2 s_1 {\color{red} s_3}s_0 s_1 s_2,\, \quad
  s_2 s_1 s_0 {\color{red} s_3}s_1s_2,\,\quad \text{ and } \quad
  s_2 s_1 s_0 s_1 {\color{red} s_3} s_2.\]
Observe that any reduced expression of $g_1$ can be obtained from another via commutation relations. Then $g_1$ is fully commutative. Also, its canonical reduced word is of case $(b)$ in Theorem \ref{thrm: type B}. 
 
Element $g_2 =[(1342);(0,0,1,1)]=[-2,3]\cdot [1,2]\cdot [-1,1]=(s_2 s_1 s_0 s_1 s_2 s_3)\cdot (s_1 s_2) \cdot (s_1 s_0 s_1)$ is not fully commutative, since its canonical reduced word is neither case $(a)$ nor $(b)$. Observe that it has a reduced expression $s_2 s_1 s_0 s_1 s_2 s_3 {\color{red} s_2 s_1 s_2} s_0 s_1$ that is not equivalent to the first one via commutation relations. Also,   $g_2 =\left[\begin{smallmatrix} &{\color{blue}1}&&\\ &&&-1\\ {\color{blue}1}&&&\\ &&{\color{blue}-1}&\end{smallmatrix}\right]$ has the pattern $(2,1,-3)$, which is one of the patterns $(a,b,c)$ where $|a|>b>c$ that fully commutative elements in $B_n$ do not have. 
 \end{example}
  
Interpreting canonical reduced words as certain plane partitions, Stembridge gave the following counting formula for fully commutative elements in $B_n$. 

 \begin{prop}[{\cite[Prop.~5.9]{S97}}]\label{prop: type B}
 In $B_n$, there are $(n+2)C_{n} -1$ fully commutative elements, where $C_n = \frac{1}{n+1} {2n \choose n}$ is the $n$th Catalan number. 
\end{prop}

\subsection{Fully commutative elements in $G(m,1,n)$}
Let $s_1, s_2, \ldots, s_{n-1}$ be simple transpositions, i.e., $s_j = [(j\, j+1); (0,\ldots, 0)]$ and $s_0 = [\id; (1, 0, \ldots, 0)]$ be a diagonal reflection. The group $G(m,1,n)$ can be generated by $s_0, s_1, s_2, \ldots, s_{n-1}$ with defining relations:
\begin{center}
\begin{tabular}{rll}
$s_0^m = s_i^2$ & $=1$ & for $1\leq i\leq n-1$,\\
$s_i s_j$ & $= s_j s_i $ & for $i+1 < j\leq n-1$, \\
$s_{i+1}s_i s_{i+1}$ & $ = s_i s_{i+1}s_i$ & for $1\leq i \leq n-2$,\\
$s_1 s_0 s_1 s_0 $ & $= s_0 s_1 s_0 s_1$.&
\end{tabular}
\end{center}

\begin{example}
Consider $G(3,1,3)$ with generating set $\{s_0, s_1, s_2\}$ where $s_0 =[\id; (1,0,0)]$, $s_1=[(12);(0,0,0)]$ and $s_2=[(23);(0,0,0)]$. Element $[\id; (2,1,0)]$ is not fully commutative since it has two reduced expressions $s_1 s_0 s_1 (s_0)^2 $ and $(s_0)^2 s_1 s_0 s_1$, where neither has adjacent commuting generators. So one cannot be obtained from the other through commutation relations, making this element not fully commutative. Element $[(13); (1,1,1)]$ is fully commutative with reduced expressions $s_0 s_1 {\color{red} s_0 s_2 }s_1 s_0=s_0 s_1 {\color{red} s_2 s_0 }s_1 s_0$.
\end{example}

By focusing on certain prefixes and suffixes of reduced expressions, \cite{FKLO}\footnote{When comparing the present work with \cite{FKLO}, the reader should note that we follow Stembridge's choice of generators, which are different from the ones used in \cite{FKLO}.} showed the following counting formula that agrees with Proposition~\ref{prop: type B} when $m=2$.
\begin{theorem}[{\cite[Cor.~4.12]{FKLO}}]\label{thrm: FKLO}
For $n\geq 3$, the number of fully commutative elements in $G(m,1,n)$ is equal to 
\[
 m(m-1)\sum\limits_{s=0}^{n-2} \frac{(n+s)! (n-s+1)}{s!(n+1)!} m^{n-2-s} + (2m-1)C_n -(m-1).
\]
\end{theorem}

In his Ph.D. thesis, Mak presented certain coset representatives, similar to what Stembridge used in $B_n$.
\begin{proposition}[{\cite[Prop.~2.2.6]{Mak}}]\label{prop: Mak}
The shortest left coset representatives for $G(m,1,n)/G(m,1,n-1)$ are \begin{multline*} \{s_0^{\epsilon} s_1\cdots s_{n-1}, \quad s_1s_0^{\epsilon}s_1\cdots s_{n-1}, \quad \ldots, \quad s_{n-1}\ldots s_1s_0^{\epsilon} s_1\ldots s_{n-1}\, | \, 1\leq \epsilon\leq m-1\} \\
\cup\, \{ 1, \quad s_{n-1},\quad  s_{n-2}s_{n-1}, \quad \ldots,\quad s_{n-1}\cdots s_1\}.\qquad \qquad\qquad\qquad\end{multline*}
\end{proposition}

For $0< i\leq j$, $k\geq 0$, and $\epsilon\in \{1, 2, \ldots, m-1\}$, let $[i^\epsilon,j]=[i,j]=s_i\cdot s_{i+1}\cdots s_j$, $[(- i)^{\epsilon}, j]=s_i\cdot s_{i+1}\cdots s_1  s_0^{\epsilon} s_1\cdot s_2\cdots s_j$ and $[0^{\epsilon}, k]=s_{0}^{\epsilon} s_1 \cdots s_k$. As a result of Proposition \ref{prop: Mak}, it follows that there exists a canonical reduced word for every element in $G(m,1,n)$.

\begin{definition}\label{canonical word}
In $G(m,1,n)$, every element has a canonical reduced word \[[m_1^{a_1}, n_1]\cdot [m_2^{a_2}, n_2]\cdots [m_r^{a_r}, n_r],\] where $n>n_1>\cdots >n_r\geq 0$, $|m_i|\leq n_i$ and $1\leq a_i\leq m-1$. 
\end{definition}

 \begin{example}\label{example}
 Consider $G(7,1,6)$, and let $g=\left[\begin{smallmatrix} &\omega^2&&&&&\\ &&1&&&\\ \omega&&&&&\\ &&&\omega^4&&\\ &&&&&\omega^6\\ &&&&\omega^5&\end{smallmatrix}\right]=[(132)(4)(56);(1, 2,0,4,5, 6)].$ Then $g$ has a canonical reduced word
 \[ [-4^6, 5]\cdot [-4^5, 4]\cdot [-3^4,3]\cdot [2^1,2]\cdot [0^2,1]\cdot [0^1, 0]\]
with a reduced expression\footnote{Here canonical reduced word and reduced expression refer to the same expression. We call the one with bracket notation \emph{canonical reduced word} and the one with generators multiplied together \emph{reduced expression}.} 
\[
s_4 s_3 s_2 s_1 s_0^6 s_1 s_2 s_3 s_4 s_5 \cdot
s_4 s_3 s_2 s_1 s_0^5 s_1 s_2 s_3 s _4 \cdot
s_3 s_2 s_1 s_0^4 s_1 s_2 s_3 \cdot
 s_2 \cdot
s_0^2 s_1 \cdot
s_0^1.
\]
Observe that the exponent $a_i$ in every block $[m_i^{a_i}, n_i]$ where $m_i \leq 0$ corresponds to the weight of a nontrivial entry in $g$. 

\section{The main theorem}
\label{sec:main}

Before we state our main theorem, we need a few definitions first.

\begin{definition}
Let $g\in G(m,1,n)$. We say $g$ has a \emph{nontrivial} entry if that entry is neither $0$ nor $1$.
\end{definition}

\begin{definition}
For positive integers $m$ and $n$, define a map 
\[
\pi: G(m, 1, n) \rightarrow G(2, 1, n) = B_n
\]
where $\pi(g)$ is the result of replacing every nontrivial entry in the matrix of $g$ with $-1$. 
\end{definition}

Continuing with Example \ref{example}, the image $\pi(g)=\left[\begin{smallmatrix} &-1&&&&&\\ &&1&&&\\ -1&&&&&\\ &&&-1&&\\ &&&&&-1\\ &&&&-1&\end{smallmatrix}\right]=[(132)(56);(1, 1,0,1,1, 1)]\in G(2,1,6)$ has a canonical reduced word
 \[ [-4, 5]\cdot [-4, 4]\cdot [-3,3]\cdot [2,2]\cdot [0,1]\cdot [0, 0]\]
with a reduced expression 
\[
s_4 s_3 s_2 s_1 s_0 s_1 s_2 s_3 s_4 s_5 \cdot
s_4 s_3 s_2 s_1 s_0 s_1 s_2 s_3 s _4 \cdot
s_3 s_2 s_1 s_0 s_1 s_2 s_3 \cdot
s_2 \cdot
s_0 s_1 \cdot
s_0.
\]
\end{example}
 
 \begin{theorem}\label{main theorem}
Let $g\in G(m,1,n)$. It is fully commutative if and only if $\pi(g)$ is fully commutative in $G(2,1,n)$.
\end{theorem}

To prove Theorem \ref{main theorem}, we need a few propositions. The first proposition describes how the map $\pi$ works on canonical reduced words. In general, one does not have $\pi(g_1 g_2) = \pi(g_1)\pi(g_2)$ for an arbitrary pair $g_1, g_2 \in G(m,1,n)$. So $\pi$ is not a group homomorphism from $G(m,1,n)$ to $G(2,1,n)$. But it does behave nicely with respect to canonical reduced words.
 
  \begin{proposition}\label{prop: obs}
 Let $g\in G(m,1,n)$ have a canonical reduced word $[m_1^{a_1}, n_1]\cdot [m_2^{a_2}, n_2]\cdots [m_r^{a_r}, n_r]$. Then $\pi(g) \in G(2,1,n)$ has a canonical reduced word
\[\begin{array}{rl}
 \pi([m_1^{a_1}, n_1]\cdot [m_2^{a_2}, n_2]\cdots [m_r^{a_r}, n_r])&=\pi([m_1^{a_1}, n_1])\cdot \pi([m_2^{a_2}, n_2])\cdots \pi([m_r^{a_r}, n_r])\\ 
 & = [m_1, n_1]\cdot [m_2, n_2]\cdots [m_r, n_r].
 \end{array}\]
 \end{proposition}

\begin{proof}
If all $m_i$'s in the canonical reduced word $[m_1^{a_1}, n_1]\cdot [m_2^{a_2}, n_2]\cdots [m_r^{a_r}, n_r]$ of $g$ are all positive, then $\pi(g)$ and $g$ have identical canonical reduced word. Then we are done.

As in Example \ref{example}, if $m_i$ of the $i$-th block $[m_i^{a_i}, n_i]$ from the left in the canonical reduced word is non-positive, then its reduced expression has the term $s_0^{a_i}$, which corresponds to the $i$-th nontrivial entry from the right columnwise in $g$ with weight $a_i$. Then its image $\pi(g)$ has a $-1$ at the same location. Then the canonical reduced word of $\pi(g)$ has a block $[m_i, n_i]$ whose reduced expression has the term $s_0$, as desired.
\end{proof}

The next proposition spells out the full commutativity condition on canonical reduced words in $G(m,1,n)$. 

\begin{prop}\label{main prop}
Let $g\in G(m,1,n)$. Then $g$ is fully commutative if and only if its canonical reduced word $[m_1^{a_1}, n_1]\cdot [m_2^{a_2}, n_2]\cdots [m_r^{a_r}, n_r]$ has either
\begin{enumerate}[(a)]
\item $m_1 >\cdots >m_{r-1}>-m_r>0$, or
\item for some $s\leq r$,  $m_1>\cdots>m_s=m_{s+1} = \cdots = m_r = 0$.
\end{enumerate}
\end{prop}

\begin{proof}
$(\Rightarrow)$ Assume that $[m_1^{a_1}, n_1]\cdot [m_2^{a_2}, n_2]\cdots [m_r^{a_r}, n_r]$ is the canonical reduced word for some fully commutative element $g\in G(m,1,n)$. By Definition \ref{canonical word}, $n>n_1 >\cdots >n_r\geq 0$ and $|m_i|\leq n_i$. 
\begin{enumerate}
\item Case $(a)$: Let $a,b \in \{1, \ldots, m-1\}$. For $i>0$, $[-1^a,i]\cdot {\color{red} s_0^b} = s_1\cdot s_0^a\cdot s_1\cdots s_i\cdot{\color{red} s_0^b}= s_1\cdot s_0^a\cdot s_1\cdot {\color{red} s_0^b}\cdot [2,i]$. Note that $s_1s_0^a s_1 s_0^b=s_0^b s_1 s_0^a s_1$, but these two expressions are not equivalent via commutation relations. So a fully commutative element cannot have reduced expressions that contain $s_1s_0^a s_1 s_0^b$ or $s_0^b s_1 s_0^a s_1$. 

For $i>j>0$, $[-1^a,i]{\color{red} s_j}=s_1\cdot s_0^a\cdot s_1\cdots s_i\cdot {\color{red} s_j}=s_1\cdot s_0^a\cdot s_1\cdots s_{j-1}\cdot s_j\cdot s_{j+1}\cdot {\color{red} s_j}\cdot s_{j+2}\cdots s_i$. Note that here $s_j s_{j+1} s_j = s_{j+1} s_j s_{j+1}$. Since one cannot be achieved from the other via commutation relations, a fully commutative element cannot have reduced expressions that contain $s_j s_{j+1} s_j$ or $s_{j+1} s_j s_{j+1}$. 

Since $|m_{i+1}|<n_{i+1}< n_{i}$ by Definition \ref{canonical word},  $[-1^a, i]s_0^b$ or $[-1^a, i]s_j$ occurs in $[m_i^{a_i}, n_i][m_{i+1}^{a_{i+1}}, n_{i+1}]$ when $m_i<0$. Then we must have $m_1, \ldots, m_{r-1}\geq 0$.

\item Case $(b)$: For $j>k\geq i\geq 0$, $[i^a,j]{\color{red} s_k}=[i^a, k-1]s_k s_{k+1} {\color{red} s_k}[k+2,j]$. A fully commutative element cannot have canonical reduced words that contain $[i^a, j]s_k$, since $s_k s_{k+1} s_k = s_{k+1} s_k s_{k+1}$ when $k>0$. In $[m_i^{a_i}, n_i][m_{i+1}^{a_{i+1}}, n_{i+1}]$, $[i^a, j]s_k$ occurs when $|m_{i+1}|\geq |m_i|$. To avoid having $s_k s_{k+1} s_k = s_{k+1} s_k s_{k+1}$ $(k>0)$, we need $|m_{i+1}|<|m_{i}|$ or $m_i = m_{i+1}=0$ for $1\leq i <r$.
\end{enumerate}
Thus, for $g\in G(m,1,n)$ a fully commutative element, its canonical reduced word is either of case $(a)$ or $(b)$.

$(\Leftarrow)$ We prove the contrapositive statement: if $g\in G(m,1,n)$ is not fully commutative, then its canonical reduced word is neither case $(a)$ nor $(b)$. Suppose $g\in G(m,1,n)$ is not fully commutative. Then there is a reduced expression for $g$ containing at least one of the terms
\[s_1 s_0^{a} s_1 s_0^{a'}, \quad s_0^{a}s_1s_0^{a'}s_1, \quad s_{i+1}s_i s_{i+1}, \, \text{       and       } \,s_i s_{i+1} s_i  \quad (i\geq 1, \text{ and } a,a'\in \{1, \ldots m-1\})\]
 that is equivalent, via commutation relations, to the reduced expression associated with the canonical reduced word of $g$. 
 
 We now show every reduced expression containing such a term leads to a canonical reduced word that is neither case $(a)$ nor $(b)$. We discuss them by the following cases. 
 \begin{enumerate}
 \item $s_{i} s_{i+1} s_{i}$: These three factors cannot be in the same block due to the structure of an individual block: the indices of the factors are either strictly increasing or strictly decreasing then strictly increasing. The three factors also cannot be in all distinct blocks without breaking the rule $n>n_1>\cdots>n_r\geq 0$ and $|m_i|\leq n_i$ by Definition \ref{canonical word}. 
 
 Then assume these factors are in two adjacent blocks, which leads to two cases.
 \begin{enumerate}
 \item Assume $s_i s_{i+1}$ is in block $[m_1^{a_1}, n_1]$ and $s_i$ is in block $[m_2^{a_2}, n_2]$. In order to move $s_i$ in block $[m_2^{a_2}, n_2]$ next to $s_i s_{i+1}$ in block $[m_1^{a_1}, n_1]$ using only commutation relations, there should be no factor to the left of $s_i$ in block $[m_2^{a_2}, n_2]$. This means that $s_i$ is the first factor in block $[m_2^{a_2}, n_2]$. Then $m_2$ must be $i$ or $-i$ by Definition \ref{canonical word}. Similarly, the factors to the right of $s_i s_{i+1}$ in block $[m_1^{a_1}, n_1]$, if exist, must have indices larger than $i+1$. This means that $n_1 \geq i+1$. Since $s_i s_{i+1}$ is in block $[m_1^{a_1}, n_1]$, the indices of factors to the left of $s_is_{i+1}$ can be strictly increasing to $i-1$ or strictly decreasing to 0 then strictly increasing to $i-1$. Then $m_1\leq i$ by Definition \ref{canonical word}. Then either $m_1<0$ (while $r\geq 2$) or $|m_1|\leq |m_2|$, which is neither case $(a)$ nor $(b)$. 
 
  \item Assume $s_i$ is in block $[m_1^{a_1}, n_1]$ and $s_{i+1}s_i$ is in block $[m_2^{a_2}, n_2]$. In order to move $s_i$ from block $[m_1^{a_1}, n_1]$ next to $s_i s_{i+1}$ from block $[m_2^{a_2}, n_2]$ using only commutation relations, there should be no factors to the right of $s_i$ in block $[m_1^{a_1}, n_1]$. This means that $s_i$ is the last factor in block $[m_1^{a_1}, n_1]$. Thus $n_1 =i$ by Definition \ref{canonical word}. Similarly, the indices of the factors to the left of $s_{i+1} s_i$ in block $[m_2^{a_2}, n_2]$, if exist, should be larger than $i+1$. This means that $0>-(i+1)\geq m_2$ by Definition \ref{canonical word}. Since $|m_1|<n_1=i$ and $|m_2|\geq i+1$, then $|m_1|< |m_2|$, which is neither case $(a)$ nor $(b)$.
 \end{enumerate}
 
 \item $s_{i+1} s_{i} s_{i+1}$: As in case $(1)$, these three factors cannot be in all distinct blocks without breaking the rule $n>n_1>\cdots>n_r\geq 0$ and $|m_i|\leq n_i$. Also, these three factors cannot be in the same block. Suppose they are in the same block. Then we must have a reduced expression containing $s_{i+1} s_i s_{i-1} \cdots s_1 s_0^{a_1} s_1 \cdots s_{i-1} s_i s_{i+1}$ by Definition \ref{canonical word}. Thus, it is impossible to have $s_{i+1} s_i s_{i+1}$ via commutation relations. It suffices to assume the three factors are in two adjacent blocks, which also leads to two cases.
 
 \begin{enumerate}
 \item Assume $s_{i+1} s_i$ is in block $[m_1^{a_1}, n_1]$ and $s_{i+1}$ is in block $[m_2^{a_2}, n_2]$. In order to move $s_{i+1}$ from block $[m_2^{a_2}, n_2]$ next to $s_{i+1} s_{i}$ from block $[m_1^{a_1}, n_1]$ using only commutation relations, there should be no factors to the left of $s_{i+1}$ in block $[m_2^{a_2}, n_2]$. This means that $s_{i+1}$ is the first factor in block $[m_2^{a_2}, n_2]$. However, the indices of the factors to the right of $s_{i+1} s_i$ in block $[m_1^{a_1}, n_1]$ must be strictly decreasing to $0$ then strictly increasing to at least $i+1$ by Definition \ref{canonical word}. This means there is another $s_{i+1}$ to the right of $s_{i+1}s_i$ in block $[m_1^{a_1}, n_1]$. Thus, it is impossible to move $s_{i+1}$ from block $[m_2^{a_2}, n_2]$ next to $s_{i+1}s_i$ from block $[m_2^{a_2}, n_2]$ using only commutation relations. Then this case does not exist.

 \item Assume $s_{i+1}$ is in block $[m_1^{a_1}, n_1]$ and $s_i s_{i+1}$ is in block $[m_2^{a_2}, n_2]$. In order to move $s_{i+1}$ from block $[m_1^{a_1}, n_1]$ next to $s_i s_{i+1}$ from block $[m_2^{a_2}, n_2]$ using only commutation relations, factors to the left of $s_i s_{i+1}$ in block $[m_2^{a_2}, n_2]$, if exist, must have indices smaller than $i$. This means $0\leq m_2 \leq i$ by Definition \ref{canonical word}. Factors to the right of $s_{i}s_{i+1}$ in block $[m_2^{a_2}, n_2]$, if exist, must have indices larger than $i+1$. Then $n_2\geq i+1$ by Definition \ref{canonical word}. Similarly, there should be no factors to the right of $s_{i+1}$ in block $[m_2^{a_2}, n_2]$. This means that $s_{i+1}$ is the last factor in block $[m_1^{a_1}, n_1]$. Then $n_1=i+1$ by Definition \ref{canonical word}. Then we have $n_1\leq n_2$, which violates the rule $n>n_1>\cdots>n_r\geq 0$. Then this case does not exist.
 \end{enumerate}
 
 \item $s_1 s_0^{a} s_1 s_0^{a'}$: 
Since the indices of $s_1 s_0^{a} s_1 s_0^{a'}$ is neither strictly increasing nor strictly decreasing and then strictly increasing, they cannot be in a single block. Below we focus on cases where the four factors come from two adjacent blocks in the canonical reduced word. Similar analysis can be applied to cases when the factors are in three blocks and all distinct blocks.

 \begin{enumerate}
 \item Assume $s_1 s_0^a s_1$ is in block $[m_1^{a}, n_1]$ and $s_0^{a'}$ is in block $[m_2^{a'}, n_2]$. Since the indices of $s_1 s_0^a s_1$ in block $[m_1^{a}, n_1]$ decreases and then increases, $m_1<0$ by Definition \ref{canonical word}. In order to move $s_0^{a'}$ in block $[m_2^{a'}, n_2]$ next to $s_1 s_0^a s_1$ in block $[m_1^a, n_1]$ using only commutation relations, there should be no factors to the left of $s_0^{a'}$ in block $[m_2^{a'}, n_2]$. Then $s_0^{a'}$ is the first factor in block $[m_2^{a'}, n_2]$. Then $m_2=0$ by Definition \ref{canonical word}. Similarly, factors to the right of $s_1 s_0^a, s_1$ in block $[m_1^a, n_1]$ must have indices larger than $1$. Then $n_1\geq 1$ by Definition \ref{canonical word}. And the indices of factors to the left of $s_1 s_0^a s_1$ in block $[m_1^a, n_1]$ must be strictly decreasing to $1$. Then $0>-1\geq m_1$ by Definition \ref{canonical word}. Then we have $m_1<m_2$, which is neither case $(a)$ nor $(b)$.
 
 \item Assume $s_1 s_0^a$ is in block $[m_1^{a}, n_1]$ and $s_1s_0^{a'}$ is in block $[m_2^{a'}, n_2]$. Since the indices of both $s_1 s_0^a$ and $s_1 s_0^{a'}$ strictly decrease, then $m_1 <0$ and $m_2<0$ by Definition \ref{canonical word}. By the same definition, factors to the right of $s_1s_0^a$ in block $[m_1^a, n_1]$ must have indices strictly increasing from $1$. Then there is a $s_1$ to the right of $s_1 s_0^a$ in block $[m_1^a, n_1]$. Thus, it is impossible to move $s_1 s_0^{a'}$ in block $[m_2^{a'}, n_2]$ next to $s_1 s_0^a$ in block $[m_1^a, n_1]$ using only commutation relations. Then this case does not exist. 
 
 \item Assume $s_1$ is in block $[m_1^{a_1}, n_1]$ and $s_0^{a} s_1 s_0^{a'}$ in block $[m_2^{a_2}, n_2]$. This is impossible since the indices of $s_0^{a} s_1 s_0^{a'}$ is neither strictly increasing nor strictly decreasing and then strictly increasing. Then this case does not exist. 
 \end{enumerate}
 
 \item $s_0^{a} s_1 s_0^{a'} s_1$: As in Case $(3)$, we discuss cases where the four factors are in two adjacent blocks in the canonical reduced word.

 \begin{enumerate}
 \item Assume $s_0^{a} s_1 s_0^{a'}$ is in block $[m_1^{a_1}, n_1]$ and $s_1$ is in block $[m_2^{a_2}, n_2]$. This is impossible since the indices of $s_0^{a} s_1 s_0^{a'}$ is neither strictly increasing nor strictly decreasing and then strictly increasing. Then this case does not exist. 
 
 \item Assume $s_0^{a} s_1$ is in block $[m_1^{a}, n_1]$ and $s_0^{a'} s_1$ is in block $[m_2^{a'}, n_2]$. In order to move $s_0^{a'} s_1$ in block $[m_2^{a'}, n_2]$ next to $s_0^a s_1$ in block $[m_1^a, n_1]$ using only commutation relations, there should be no factors to the left of $s_0^{a'}s_1$ in block $[m_2^{a'}, n_2]$. Then $s_0^{a'} s_1$ are the first two factors in block $[m_2^{a'}, n_2]$. Then $m_2 =0$ by Definition \ref{canonical word}. And factors to the right of $s_0^{a'}s_1$ in block $[m_2^{a'}, n_2]$, if exist, must have indices larger than $1$. Then $n_2\geq 1$ by Definition \ref{canonical word}. Similarly, there should be no factors to the right of $s_0^{a} s_1$ in block $[m_1^a, n_1]$, since $s_1$ and $s_2$ do not commute with each other. Then $n_1=1$ by Definition \ref{canonical word}. Then we have $n_1\leq n_2$, which violates the rule $n>n_1>\cdots>n_r\geq 0$. Then this case does not exist. 

 \item Assume $s_0^{a}$ is in block $[m_1^a, n_1]$ and $s_1 s_0^{a'} s_1$ is in block $[m_2^{a'}, n_2]$. In order to move $s_0^a$ in block $[m_1^a, n_1]$ next to $s_1 s_0^{a'} s_1$ in block $[m_2^{a'}, n_2]$ using only commutation relations, there should be no factors to the right of $s_0^a$ in block $[m_1^a, n_1]$. Then $s_0^a$ is the only factor in block $[m_1^a, n_1]$. Then $m_1 = n_1 =0$ by Definition \ref{canonical word}. Since $s_0$ commutes with every generator other than $s_1$, there is no additional restriction on factors on both sides of $s_1 s_0^{a'} s_1$ in block $[m_2^{a'}, n_2]$. Then $n_2\geq 1$. Then we have $n_1< n_2$, which violates the rule $n>n_1>\cdots>n_r\geq 0$. Then this case does not exist.
 \end{enumerate}
 \end{enumerate}
 
Therefore, if $g$ is not fully commutative, then its canonical reduced word cannot be case $(a)$ or $(b)$. 
$\qedhere$
\end{proof}

\begin{proof}[Proof of Theorem~\ref{main theorem}]
By Proposition \ref{main prop}, $g$ is fully commutative in $G(m,1,n)$ if and only if its canonical reduced word is of case $(a)$ or $(b)$.  By Proposition \ref{prop: obs}, this happens if and only if the canonical reduced word of $\pi(g)$ in $G(2,1,n)$ is also of case $(a)$ or $(b)$, which occurs if and only if $\pi(g)$ is fully commutative in $G(2,1,n)$ by Theorem \ref{thrm: type B}. 
\end{proof}

Thus, the pattern avoidance of fully commutative elements in $G(2,1,n)$ naturally extends to fully commutative elements in $G(m,1,n)$, and we list them in Table \ref{patterns}.

\begin{table}[!htp]
\centering
\scalebox{0.7}{
\begin{tabular}{|c|c|l|}
\hline
Patterns & In $G(2,1,n)$ & \text{ In $G(m,1,n)$}\\
\hline
&&\\
$(-1, -2)$ & $\left[\begin{smallmatrix} -1&\\ &-1\end{smallmatrix}\right]$ & $ \left[\begin{smallmatrix} \omega^{a} &\\ &\omega^{b}\end{smallmatrix}\right] $\text{    $(a, b \neq 0)$}\\
&&\\
$(\pm 3,2,\pm 1)$ & $\left[\begin{smallmatrix} && \pm 1\\ &1&\\ \pm 1&& \end{smallmatrix}\right] $ &$\left[\begin{smallmatrix} && \omega^a\\ &1&\\ \omega^b&& \end{smallmatrix}\right]$\\
&&\\
$(\pm 3,\pm 1,-2)$ & $ \left[\begin{smallmatrix} && \pm 1\\ \pm 1&&\\ &-1& \end{smallmatrix}\right]$ & $ \left[\begin{smallmatrix} && \omega^a\\ \omega^b&&\\ &\omega^c& \end{smallmatrix}\right]$ \text{ $(c \neq 0)$} \\
&&\\
$(\pm 2,\pm 1,-3)$ & $\left[\begin{smallmatrix} & \pm 1&\\ \pm 1&&\\& &-1 \end{smallmatrix}\right]$ &$\left[\begin{smallmatrix} & \omega^a&\\ \omega^b&&\\ &&\omega^c \end{smallmatrix}\right]$ \text{ $(c \neq 0)$}\\
&&\\
$(\pm 2,-3,\pm 1)$ & $\left[\begin{smallmatrix} & \pm 1&\\ &&- 1\\ \pm 1&& \end{smallmatrix}\right]$ &$ \left[\begin{smallmatrix}& \omega^a&\\ &&\omega^b\\ \omega^c&& \end{smallmatrix}\right]$ \text{ $(b \neq 0)$}\\
&&\\
$(\pm 1,-3,-2)$ & $\left[\begin{smallmatrix} \pm 1&&\\ &&- 1\\ &-1&\end{smallmatrix}\right]$ &$ \left[\begin{smallmatrix}  \omega^a&&\\ &&\omega^b\\ &\omega^c& \end{smallmatrix}\right]$ \text{$(b,c \neq 0)$}\\
&&\\
\hline
\end{tabular}}
\vspace{0.1cm}
\caption{Pattern avoidance of fully commutative elements in $G(m,1,n)$}\label{patterns}
\end{table}

\begin{corollary}\label{cor: coeff}
In $G(m,1,n)$, the number of fully commutative elements is 
\[\sum\limits_{k=0}^{n} \alpha_{n,k} (m-1)^{k},\]
where $\alpha_{n,k}$ is the number of fully commutative elements with $k$ $-1$'s in $G(2,1,n)$.
\end{corollary}

\begin{proof}
Assume $g$ has $k$ nontrivial entries, then its image $\pi(g)\in G(2,1,n)$ has $k$ $-1$'s. Then the canonical reduced word of $\pi(g)$ has $k$ blocks $[m_i, n_i]$ where $m_i \leq 0$. This means that there are $k$ $s_0$'s in the reduced expression of $\pi(g)$. Then every $s_0$ in the reduced expression of $\pi(g)$ has $m-1$ distinct pre-images, i.e., $s_0$, $(s_0)^2$, $\ldots$, $(s_0)^{m-1}$, in the reduced expression of $g$. Therefore, $\pi(g)$ has $(m-1)^k$ distinct pre-images. By Theorem \ref{main theorem}, these pre-images are fully commutative in $G(m,1,n)$. Again by Theorem \ref{main theorem}, summing over all possible values of $k$ gives the total number of fully commutative elements in $G(m,1,n)$. 
\end{proof}

\begin{corollary}
In $G(2,1,n)$, there are
\begin{enumerate}
\item $C_{n+1} -1$ fully commutative elements with one $-1$, and
\item ${2n \choose n+k} - {2n \choose n+k+1}$ fully commutative elements with $k$ $-1$'s ($k\neq 1$). 
\end{enumerate}
\end{corollary}

\begin{proof}
By Corollary \ref{cor: coeff}, it suffices to show that in $G(m,1,n)$, the number of fully commutative elements is \begin{equation}\label{equ: formula} \sum\limits_{k= 0}^{n}\left( {2n \choose n+k} - {2n \choose n+k+1}\right) (m-1)^k + (C_{n} - 1)(m-1).\end{equation}
 When $n=2$, fully commutative elements with two nontrivial entries are of the look $\begin{bmatrix} & \star\\ \star&\end{bmatrix}$, with one nontrivial entry can have one of four looks: $\begin{bmatrix} \star &\\ &1\end{bmatrix}$, $\begin{bmatrix} 1&\\ &\star\end{bmatrix}$, $\begin{bmatrix} &1\\ \star&\end{bmatrix}$ and $\begin{bmatrix} &\star\\ 1&\end{bmatrix}$, and with no nontrivial entry are exactly the ones that are fully commutative in $\Symm_2$, and there are $C_2=2$ of them. Since each $\star$ has $m-1$ choices, the total number of fully commutative elements in $G(m,1,2)$ is 
\[
1\cdot (m-1)^2 + 4\cdot (m-1)^1 + C_2\cdot (m-1)^0 = m^2 +2m -1,
\]
which is $(\ref{equ: formula})$ evaluated when $n=2$.
 
 When $n\geq 3$, we show that $(\ref{equ: formula})$ agrees with Theorem~\ref{thrm: FKLO}, i.e., for positive integers $m$ and $n$, the following equality is true:
 {{\footnotesize \[m(m-1)\sum\limits_{s=0}^{n-2} \frac{(n+s)! (n-s+1)}{s!(n+1)!} m^{n-2-s} +(2m-1)C_n -(m-1) =\sum\limits_{k= 0}^{n} \left({2n \choose n+k} - {2n \choose n+k+1} \right)(m-1)^k + (C_{n} - 1)(m-1).\] }}

Removing identical terms on both sides, we need to show \[m(m-1)\sum\limits_{s=0}^{n-2} \frac{(n+s)! (n-s+1)}{s!(n+1)!} m^{n-2-s} +mC_n =\sum\limits_{k= 0}^{n} \left( {2n \choose n+k} - {2n \choose n+k+1}\right) (m-1)^k .\] It suffices to show that $[m^j]\mathrm{LHS} = [m^j]\mathrm{RHS}$ for $j=0,1,2, \ldots, n$ where $[m^j]$ denotes the coefficient of the term $m^j$. We do this in three cases. 

$(1)$ When $j=0$,
 $[m^0]\mathrm{LHS} =0$ and $ [m^0]\mathrm{RHS} =\sum\limits_{k=0}^{n}\left({2n \choose n+k} - {2n \choose n+k+1}\right)(-1)^k.$ Let $S = \sum\limits_{k=0}^{n} (-1)^k {2n \choose n+k}$. Then 
\begin{equation*}
\begin{split}
2S & = \left[(-1)^n {2n \choose 0} + (-1)^{n-1} {2n \choose 1} + \cdots + {2n \choose n}\right]+\left[ {2n \choose n} + (-1) {2n \choose n+1} + \cdots + (-1)^{n} {2n \choose 2n}\right]\\
& = (-1)^n \left[ {2n \choose 0} - {2n \choose 1} +\cdots +{2n \choose n}\right] + {2n \choose 2n}= {2n \choose n}.\\
\end{split}
\end{equation*}
 Similarly, let $T = \sum\limits_{k=0}^{n} (-1)^k {2n \choose n+k+1} $. Then 
 \begin{equation*}
\begin{split}
2T & = \left[(-1)^{n-1} {2n \choose 0} + (-1)^{n-2} {2n \choose 1} + \cdots + {2n \choose n-1}\right]+ \left[  {2n \choose n+1} -{2n \choose n+2} + \cdots + (-1)^{n-1} {2n \choose 2n}\right]\\
& = (-1)^{n-1} \left[ {2n \choose 0} - {2n \choose 1} +\cdots +{2n \choose 2n}\right] + {2n \choose n}= {2n \choose n}.\\
\end{split}
\end{equation*}
Since $ S=T=\frac{1}{2} {2n \choose n}$, then $[m^0]\mathrm{RHS} = S-T =0=[m^0]\mathrm{LHS}$.
 
$(2)$ When $j=1$, $[m^1]\mathrm{LHS} = -\frac{(2n-2)! 3}{(n-2)!(n+1)!}+C_n= \frac{1}{n-1}{{2n-2} \choose n}$ and
\begin{equation*} 
\begin{split}
[m^1]\mathrm{RHS}& = \sum\limits_{k=0}^{n}\left({2n \choose n+k} - {2n \choose n+k+1}\right)k(-1)^{k-1} \\
 & =\sum\limits_{k=0}^n (n+k-n){2n \choose n+k}(-1)^{k-1} +\sum\limits_{k=0}^n (n+k+1-n-1){2n \choose n+k+1}(-1)^{k}\\
 & = 2n \sum\limits_{k=0}^{n}(-1)^{k-1}{2n-1 \choose n+k-1} +n \sum\limits_{k=0}^{n}(-1)^{k}{2n \choose n+k}\\ 
 &\qquad +2n \sum\limits_{k=0}^{n} (-1)^k {2n-1 \choose n+k} -(n+1)\sum\limits_{k=0}^{n} (-1)^k {2n \choose n+k+1}.\\
\end{split}
\end{equation*}

From part $(1)$, we have $\sum\limits_{k=0}^{n} (-1)^k {2n \choose n+k} = \sum\limits_{k=0}^{n} (-1)^k {2n \choose n+k+1}=\frac{1}{2}{2n \choose n}.$ Also, it is known that $\sum\limits_{i=0}^{N} (-1)^i {n \choose i}=(-1)^N {n-1 \choose N}$. Then 
\begin{equation*}
\begin{split} 
[m^1]\mathrm{RHS}& = 2n (-1)^{n-1}\left[(-1)^{n-1}{2n-2 \choose n-1}\right] + \frac{n}{2}{2n \choose n}+ 2n (-1)^{n-1}\left[ (-1)^n {2n-2 \choose n}\right] -\frac{n+1}{2}{2n \choose n}\\
& = 2n \left[ {2n-2 \choose n-1} - {2n-2 \choose n} \right]-\frac{1}{2}{2n \choose n}\\
& = \frac{2n}{n-1}{2n-2 \choose n} - \frac{2n-1}{n-1}{2n -2 \choose n}\\
& = \frac{1}{n-1} {2n-2 \choose n}\\
&=[m^1]\mathrm{LHS} .
\end{split}
\end{equation*}

$(3)$ When $j\geq 2$, $[m^j]\mathrm{LHS} =\frac{(2n-j)!(j+1)}{(n-j)!(n+1)!} - \frac{(2n-j-1)!(j+2)}{(n-j-1)!(n+1)!} = \frac{j+1}{n+1}{2n-j \choose n} - \frac{j+2}{n+1}{2n-j-1 \choose n}= \frac{j+1}{n+1}{2n-j-1 \choose n-1} - \frac{1}{n+1} {2n-j-1 \choose n}$. Since ${k\choose j} = (-1)^{k-j} {-j-1 \choose k-j} $, then
\begin{equation*}
\begin{split}
[m^j]\mathrm{RHS} &= \sum\limits_{k=0}^{n}\left({2n \choose n+k} - {2n \choose n+k+1}\right){k \choose j} (-1)^{k-j}\\
& =  \sum\limits_{k=j}^{n}\left({2n \choose n-k} - {2n \choose n-k-1}\right)\left[ (-1)^{k-j} {-j-1 \choose k-j} \right] (-1)^{k-j}\\
& = \sum\limits_{k=j}^{n}\left[ {2n \choose n-k} {-j-1 \choose k-j} - {2n \choose n-k-1}{-j-1 \choose k-j}\right].\\
\end{split}
\end{equation*}
Recall the Chu-Vandermonde identity $\sum\limits_{i=0}^{c} {a \choose i}{b \choose c-i} = {a+b \choose c}$ for general complex-valued $a$ and $b$ and any non-negative integer $c$. Then letting $c = n-j$, $i = k-j $, $a=-j-1$ and $b=2n$ gives
\[\sum\limits_{k=j}^{n}{2n \choose n-k}{-j-1 \choose k-j} = {2n-j-1 \choose n-j}\]
and letting $c=n-j-1$, $i = k-j$, $a=-j-1$ and $b=2n$ gives
\[\sum\limits_{k=j}^{n}{2n \choose n-k-1}{-j-1 \choose k-j} = 
{2n \choose n-n-1}{-j-1 \choose n-j}  + \sum\limits_{k=j}^{n-1}{2n\choose n-k-1}{-j-1 \choose k-j}=0 + {2n-j-1 \choose n-j-1}.\]

Then
\begin{equation*}
\begin{split}
[m^j]\mathrm{RHS} ={2n-j-1 \choose n-j} - {2n-j-1 \choose n-j-1}= {2n-j-1 \choose n-1} - {2n-j-1 \choose n} .
\end{split}
\end{equation*}
 
Therefore,
 \begin{multline*}
 [m^j]\mathrm{LHS}-[m^j]\mathrm{RHS} 
 =\left[\frac{j+1}{n+1}{2n-j-1 \choose n-1} - \frac{1}{n+1} {2n-j-1 \choose n}\right] \\- \left[{2n-j-1 \choose n-1} - {2n-j-1 \choose n} \right]
 =0.\qedhere
 \end{multline*}   
\end{proof}

\begin{remark}
Observe that the leading coefficient of $(\ref{equ: formula})$ is $1$. This means that the fully commutative elements in $G(m,1,n)$ with $n$ nontrivial entries have canonical reduced words $[0^{a_1}, n-1][0^{a_2}, n-2]\cdots [0^{a_{n-1}}, 1][0^{a_{n}}, 0^{a_n}]$ with the same underlying permutation, namely the reverse identity matrix.

One can also see this directly from Proposition \ref{main prop}. In order to create elements with $n$ nontrivial entries, the $m_i$'s in the canonical reduced words $[m_1^{a_1}, n_1]\cdots [m_n^{a_n}, n_n]$ must be non-positive. Since these elements are fully commutative, then their canonical reduced words are of case $(b)$ in Proposition \ref{main prop}, i.e., all $m_i$'s are $0$. Since $n>n_1>n_2>\cdots >n_n\geq 0$, then the desired elements have canonical reduced words of the form \[[0^{a_1}, n-1][0^{a_2}, n-2]\cdots [0^{a_{n-1}}, 1][0^{a_{n}}, 0^{a_n}].\]
\end{remark}

\section{Fully commutative elements in other groups}
\label{sec:open}

In this section, we discuss full commutativity in groups $G(m,m,n)$ and Shephard groups. We propose some conjectures and questions as future directions. 
\subsection{\textit{\textbf{G(m,m,n)}}}
Fully commutative elements in $G(m,m,n)$ are harder to enumerate and to characterize. The recent work by Feinberg-Kim-Lee-Oh studies elements in $G(m,m,n)$ that are fully commutative in $G(m,1,n)$. Note that such elements are not necessarily fully commutative in $G(m,m,n)$. So their counting formula does not recover Stembridge's result in $D_n$. In this section, we propose a few open questions regarding fully commutative elements in $G(m,m,n)$. 

Let $s_1, s_2, \ldots, s_{n-1}$ be simple transpositions, i.e., $s_i  =[(i\, i+1); (0, \ldots, 0)]$, and $s_{\bar{1}}= [(12); (-1,1,0,\ldots, 0)]$. For $m\geq 2$ and $n\geq 3$, the group $G(m,m,n)$ can be generated by $s_{\bar{1}}, s_1, s_2, \ldots, s_{n-1}$ with defining relations \cite{BMR}:
\begin{center}
\begin{tabular}{rll}
$(s_1s_{\bar{1}})^m= s_{\bar{1}}^2 = s_i^2 $ &$=1$ & for $1\leq i\leq n-1$,\\
$s_i s_j$ & $= s_j s_i $ & for $i+1 < j\leq n-1$, \\
$s_i s_{\bar{1}}$ & $ = s_{\bar{1}} s_i$& for $1< i \leq n-1$,\\
$s_{i+1}s_i s_{i+1}$ &$ = s_i s_{i+1}s_i$ & for $1\leq i \leq n-2$,\\
$s_{\bar{1}} s_2 s_{\bar{1}}$ & $= s_2 s_{\bar{1}} s_2$&\\
$(s_{\bar{1}} s_1 s_2)^2 $ & $ = (s_2 s_{\bar{1}} s_1)^2 $.&
\end{tabular}
\end{center}
We refer to this set of generators  $S^c=\{s_{\bar{1}}, s_1, s_2, \ldots, s_{n-1} \}$ as the \emph{classical generating set}. 

Stembridge characterized fully commutative elements by pattern avoidance in $D_n=G(2,2,n)$ with the classical generating set. Equivalent result was also obtained by Fan \cite[$\S 7$]{Fan}.

\begin{theorem}[{\cite[Thrm.~10.1]{S97}}]\label{thrm: type D}
For $g\in D_{n}$, the following are equivalent.
\begin{enumerate}
\item $g$ is fully commutative.
\item $g$ avoids all patterns $(a, b,c)$ such that $|a|>b>c$ or $-b>|a|>c$.
\end{enumerate}
\end{theorem}

There is also a counting formula for fully commutative elements in $D_n$, with the classical generating set, obtained independently by Stembridge \cite{S97} and Fan \cite{Fan}.
\begin{proposition}[{\cite[Prop.~10.4]{S97}}]
In $D_n$, there are $\frac{n+3}{2}C_n -1$ fully commutative elements, where $C_n=\frac{1}{n+1}{2n \choose n}$ is the $n$th Catalan number.
\end{proposition}

\subsubsection{\textbf{Enumeration}}
We consider counting fully commutative elements in $G(m,m,n)$ with the classical generating set. 

\begin{prop}\label{prop: mm2}
When $n=2$, there are 
\begin{enumerate}
\item $4$ fully commutative elements in $G(2,2,2)$, i.e., every element is fully commutative in $G(2,2,2)$. 
\item $2m-1$ fully commutative elements in $G(m,m,2)$ when $m\geq 3$.
\end{enumerate}
\end{prop}

\begin{proof}
\begin{enumerate}
\item Clear.
\item We know that $G(m,m,2)$ is generated by $s_1 = [\id; (0,0)]$ and $s_{\bar{1}} = [(12); (-1,1)]$. There is exactly one element that is not fully commutative element:$[\id; (d, d)]$ if $m=2d$ or $[(12); (d, -d)]$ if $m=2d+1$, with reduced expressions $s_1 s_{\bar{1}} s_1 s_{\bar{1}} \cdots = s_{\bar{1}} s_1 s_{\bar{1}} s_1 \cdots$. Since there are $2m$ elements in $G(m, m,2)$, then $2m-1$ of them are fully commutative. \qedhere
\end{enumerate} 
\end{proof}

\begin{proposition}
In $G(m,m,3)$, an element is fully commutative if and only if it has unique reduced expression.
\end{proposition}
\begin{proof}
When $n=3$, $G(m,m,3)$ is generated by $s_1=[(12); (0, 0, 0)]$, $s_2=[(23); (0, 0, 0)]$ and $s_{\bar{1}} = [(12); (-1,1,0)]$. Since no two of the generators commute, the result follows.
\end{proof}

The present mapping method, from $G(m,1,n)$ to $G(2,1,n)$, is not very useful in counting fully commutative elements in $G(m,m,n)$. At least one obstacle comes from mapping an element in $G(m,m,n)$ to an element in $G(2,2,n)$, i.e., replacing every nontrivial entry with $-1$. Consider $[(123); (1, 2,0)]$ in $G(3,3,3)$, which is fully commutative with unique reduced expression $s_1 s_{\bar{1}} s_2 s_1$. The image $[(123);(1,1,0)]$ is not fully commutative in $G(2,2,3)$, since it has reduced expressions $s_1 s_{\bar{1}} s_2 s_1=s_{\bar{1}} s_1 s_2 s_1=s_{\bar{1}} s_2 s_1 s_2$. Another challenge concerns elements such as $[(23);(1,1,1)]\in G(3,3,3)$, which is fully commutative with unique reduced expression $s_1 s_{\bar{1}} s_2 s_1s_{\bar{1}}$. However, the image $[(23);(1,1,1)]$ does not exist in $G(2,2,3)$. 

In Table \ref{classical}, we enumerate fully commutative elements in $G(m,m,3)$ for small values of $m$ and list them by length $\ell$, i.e., the minimum number of generators needed in their reduced expressions.

\begin{table}[!htbp]
\centering
\scalebox{0.5}{
\begin{tabular}{|cccccccccc|}
\hline $\ell$ & $m=2$&$m=3$ & $m=4$ & $m=5$ & $m=6$ & $m=7$ & $m=8$ & $m=9$ & $m=10$ \\
\hline 0 &1& 1 & 1 & 1 & 1 & 1 & 1 & 1 & 1 \\
\hline 1 & 3&3 & 3 & 3 & 3 & 3 & 3 & 3 & 3 \\
\hline 2 & 5&6 & 6 & 6 & 6 & 6 & 6 & 6 & 6 \\
\hline 3 & 4&6 & 8 & 8 & 8 & 8 & 8 & 8 & 8 \\
\hline 4 &1& 6 & 10 & 12 & 12 & 12 & 12 & 12 & 12 \\
\hline 5 && 6 & 12 & 16 & 18 & 18 & 18 & 18 & 18 \\
\hline 6 && & 10 & 16 & 20 & 22 & 22 & 22 & 22 \\
\hline 7 && & & 16 & 22 & 26 & 28 & 28 & 28 \\
\hline 8 && & & 2 & 18 & 24 & 28 & 30 & 30 \\
\hline 9 && & & & 4 & 26 & 32 & 36 & 38 \\
\hline 10 && & & & & 10 & 26 & 32 & 36 \\
\hline 11 && & & & & & 14 & 36 & 42 \\
\hline 12 && & & & & & & 18 & 34 \\
\hline 13 && & & & & & & 4 & 24 \\
\hline 14 & && & & & & & & 8 \\
\hline
\hline total &14 &28& 50&80&112 &156 & 198&254 &310 \\
\hline
\end{tabular}
}
\vspace{0.1cm}
\caption{f.c. elements in $G(m,m,3)$ with classical generating set}
\label{classical}
\end{table}

\vspace{-0.3in}
 
\subsubsection{\textbf{Pattern avoidance}}

We now consider characterizing fully commutative elements with the classical generating set by pattern avoidance in $G(m,m,n)$. Theorem \ref{thrm: type D} implies that a fully commutative element in $G(2,2,n+1)$ does not contain a non fully commutative element in $G(2, 2, n)$, as a submatrix. This behavior does not extend to $m>2$. For example,  $[(12); (d, -d)]$ is not fully commutative in $G(2d+1, 2d+1, 2)$, by Proposition \ref{prop: mm2}, but $[(23); (0, d, -d)]$ is fully commutative in $G(2d+1, 2d+1, 3)$ $(d\geq 2)$, with unique reduced expression $s_{\bar{1}} s_2 (s_1 s_{\bar{1}})^{d-1} s_1 s_2 s_{\bar{1}}$. 

Furthermore, one can find this kind of fully commutative elements in both $G(3,3,4)$ and $G(4,4,4)$, as listed in Table \ref{strange}. In fact, Table \ref{strange} lists all such elements in both groups. Current evidence suggests that the number of such strange fully commutative elements in $G(m,m,n+1)$ is very small.

\begin{center}
\begin{table}[!htbp]
\scalebox{0.65}{
\begin{tabular}{|lll|l|}
\hline
\text{ $G(3,3,4)$: fully commutative } && & \text{ $G(3,3,3)$: non fully commutative }\\
$\left[ \begin{smallmatrix} &1&&\\ &&&\omega^2 \\ &&\omega^2 & \\ \omega^2&&&\end{smallmatrix}\right]$  & &
 $\left[ \begin{smallmatrix} 1&&&\\ &&&\omega^2\\ &&\omega^2&\\ & \omega^2 &&\end{smallmatrix}\right]$  &  $\left[ \begin{smallmatrix} &&\omega^2\\ &\omega^2&\\ \omega^2&&\end{smallmatrix}\right]$\\
 $s_{\bar{1}} s_1 s_2 s_{\bar{1}} s_3 s_1 s_2 s_{\bar{1}}$ &  &
 $s_{\bar{1}} s_1 s_2 s_{\bar{1}} s_3 s_1 s_2 s_{\bar{1}} s_1$ &
 $s_{\bar{1}} s_1  s_2 s_{\bar{1}} s_2= s_{\bar{1}} s_1 s_{\bar{1}} s_2 s_{\bar{1}}$\\
&&&\\
$\left[ \begin{smallmatrix} &&&\omega\\ 1&&&\\ &&\omega&\\ &\omega&&\end{smallmatrix}\right]$ & &
$\left[ \begin{smallmatrix} 1&&&\\ &&&\omega\\ &&\omega&\\  &\omega&&\end{smallmatrix}\right]$ &
$\left[ \begin{smallmatrix} &&\omega\\ &\omega&\\ \omega&&\end{smallmatrix}\right]$\\
$s_{\bar{1}} s_2 s_1 s_{\bar{1}} s_3 s_2 s_1 s_{\bar{1}}$ & &
$ s_1 s_{\bar{1}} s_2 s_1 s_{\bar{1}} s_3 s_2 s_1 s_{\bar{1}}$&
$s_{\bar{1}} s_2  s_1 s_{\bar{1}} s_1=s_{\bar{1}} s_2 s_{\bar{1}} s_1 s_{\bar{1}}$\\
\hline
\text{ $G(4,4,4)$: fully commutative } && & \text{ $G(4,4,3)$: non fully commutative }\\
$\left[\begin{smallmatrix} &&&\omega^2 \\ 1&&&\\ &&\omega&\\ &\omega&&\end{smallmatrix}\right]$ &&&
$\left[\begin{smallmatrix} &&\omega^2 \\ &\omega&\\ \omega&&\end{smallmatrix}\right]$\\
$s_{\bar{1}} s_1 s_{\bar{1}} s_2 s_1 s_{\bar{1}} s_3 s_2 s_1 s_{\bar{1}}$ &&&
$s_1 s_{\bar{1}} s_2 s_1 s_{\bar{1}} s_1 s_2=s_{\bar{1}} s_1 s_{\bar{1}} s_2 s_1 s_{\bar{1}}s_1$\\
&&&\\
$\left[\begin{smallmatrix} &1&&\\ &&&\omega^3\\ &&\omega^3 &\\ \omega^2&&&\end{smallmatrix}\right]$ &&&
$\left[\begin{smallmatrix} &&\omega^3\\ &\omega^3&\\ \omega^2&&\end{smallmatrix}\right]$\\
$s_{\bar{1}} s_1 s_2 s_{\bar{1}} s_3 s_1 s_2 s_{\bar{1}} s_1 s_{\bar{1}}$ &&&
$s_{1} s_{\bar{1}} s_1 s_2 s_{\bar{1}} s_1 s_{\bar{1}}=s_2 s_1 s_{\bar{1}} s_1 s_2 s_{\bar{1}} s_1$\\
\hline
\end{tabular}}
\vspace{0.1cm}
\caption{f.c. elements that contain a non f.c. element}
\label{strange}
\end{table}
\end{center}

Comparing Theorem \ref{thrm: type D} with Theorem \ref{thrm: type B}, one notices that the pattern avoidance in $G(2,2,n)$ is also part of the pattern avoidance in $G(2,1,n)$ $(n\geq 3)$. This means that a fully commutative element in $G(2,2,n)$ is also fully commutative in $G(2,1,n)$.

This is not true in $G(m,m,n)$ when $m>2$. One can find elements that are fully commutative in $G(3,3,3)$ but not fully commutative in $G(3,1,3)$, such as $[\id; (1,2,0)]$ which is fully commutative in $G(3,3,3)$ with unique reduced expression $s_{\bar{1}}s_1$, but it is not fully commutative in $G(3,1,3)$ with reduced expressions $s_1 s_0^2 s_1 s_0$ and $s_0 s_1 s_0^2 s_1$.

Since Stembridge's result in $D_n=G(2,2,n)$ cannot be directly extended to $G(m,m,n)$, full commutativity in $G(m,m,n)$ with the classical generating set remains to be studied further.

\subsubsection{\textbf{Other generating sets}}
There are many possible choices of generating set for $G(m,m,n)$, none of which are Coxeter-like in the same way as in $G(m,1,n)$ (see the discussion preceding Lemma 4.2 in \cite{ChaDou} and \cite[$\S 3.7.2$]{Wil19}). Likely because of this phenomenon, full commutativity in $G(m,m,n)$ with the classical generating set is more challenging to investigate. In this section, we look at two other generating sets for $G(m,m,n)$.

\paragraph{\textbf{Affine generating set}}
Let $s_1$, $s_2$, \ldots, $s_{n-1}$ be simple transpositions, i.e., $s_i = [(i\, i+1); (0,\ldots, 0)]$, and $\widetilde{s}_n = [(1\, n); (-1, 0, \ldots, 0, 1)]$. Then $G(m,m,n)$ can be generated by $s_1$, $s_2$, $\ldots$, $s_{n-1}$, $\widetilde{s}_n$ with defining relations \cite{Shi-generic}: 
\begin{center}
\begin{tabular}{rll}
$(\widetilde{s}_n (s_1 s_2\cdots s_{n-1} \cdots s_2 s_1))^m=(\widetilde{s}_n)^2 = s_i^2 $ &$=1$ & for $1\leq i\leq n-1$,\\
$s_i s_j$ & $= s_j s_i $ & for $i+1 < j\leq n-1$, \\
$s_i \widetilde{s}_n$ & $ = \widetilde{s}_n s_i$& for $1<i < n-1$,\\
$s_{i+1}s_i s_{i+1}$ &$ = s_i s_{i+1}s_i$ & for $1\leq i \leq n-2$,\\
$s_i \widetilde{s}_n s_i$ & $= \widetilde{s}_n s_i \widetilde{s}_n$ & for $i=1,\, n-1$.\\
\end{tabular}
\end{center}
This generating set for $G(m,m,n)$ is the projection of the Coxeter generating set for the affine symmetric group $\widetilde{A}_n$. We refer to this set of generators $\widetilde{S}= \{ s_1, \ldots, s_{n-1}, \widetilde{s}_n\}$ as the \emph{affine generating set} of $G(m,m,n)$. 

Tables \ref{enumeration} and \ref{affine in 4} show the current enumerative evidence of fully commutative elements in $G(m,m,3)$ and $G(m,m,4)$. 

\begin{table}[!htbp]
\parbox{.45\linewidth}{
\centering
\scalebox{0.9}{
\begin{tabular}{|cccccc|}
\hline$\ell$  & $m=2$& $m=3$ & $m=4$ & $m=5$ & $m=6$ \\
\hline 0 &1& 1 & 1 & 1 & 1 \\
\hline 1 &3&3 & 3 & 3 & 3 \\
\hline 2 &6&6 & 6 & 6 & 6 \\
\hline 3 &  6& 6 & 6 & 6 & 6 \\
\hline 4 && 6 & 6 & 6 & 6 \\
\hline 5 && 6 & 6 & 6 & 6 \\
\hline 6 && & 6 & 6 & 6 \\
\hline 7 && & 6 & 6 & 6 \\
\hline 8 && & & 6 & 6 \\
\hline 9 && & & 6 & 6 \\
\hline 10  &&& & & 6 \\
\hline 11 && & & & 6 \\
\hline
\hline total &16&28 &40 & 52& 64 \\
\hline
\end{tabular}}
\vspace{0.1cm}
\captionsetup{format=myformat}
    \caption{f.c. elements in $G(m,m,3)$\newline with affine generating set}
\label{enumeration}
}
\hfill
\parbox{.45\linewidth}{
\centering
\scalebox{0.7}{
\begin{tabular}{|ccccc|}
\hline$\ell$ & $m=2$ & $m=3$ & $m=4$ & $m=5$ \\
\hline 0 & 1 & 1 & 1 & 1 \\
\hline 1 & 4 & 4 & 4 & 4 \\
\hline 2 & 10 & 10 & 10 & 10 \\
\hline 3 & 16 & 16 & 16 & 16 \\
\hline 4 & 18 & 18 & 18 & 18 \\
\hline 5 & 16 & 16 & 16 & 16 \\
\hline 6 & 10 & 18 & 18 & 18 \\
\hline 7 & &16 & 16 & 16 \\
\hline 8 && 18 & 18 & 18 \\
\hline 9 && 8 & 16 & 16 \\
\hline 10 & &10 & 18 & 18 \\
\hline 11 & && 16 & 16 \\
\hline 12 & && 10 & 18 \\
\hline 13 && & 8 & 16 \\
\hline 14 & & &10 & 18 \\
\hline 15 &&& & 8 \\
\hline 16 &&& & 10 \\
\hline 17 & &&& 8 \\
\hline 18 &&& & 10 \\
\hline
\hline total &75&135&195 &255  \\
\hline
\end{tabular}}
\vspace{0.1cm}
\captionsetup{format=myformat}
    \caption{f.c. elements in $G(m,m,4)$\newline with affine generating set}
\label{affine in 4}
}
\end{table}
\vspace{-0.3cm}
Preliminary data suggests the following conjectures.

\begin{conjecture}
Let $a_m$ denote the number of fully commutative elements in $G(m,m,n)$ $(n\geq 3)$ with the affine generating set. Then $a_{m+1}=a_{m} +k$, where $k$ is a positive integer. In particular, $k=12$ when $n=3$ and $k=60$ when $n=4$.
\end{conjecture}

\begin{conjecture}
Consider the group $G(m,m,3)$ with the affine generating set. The number of fully commutative elements of length $\ell$ when $\ell>1$ is $6$.
\end{conjecture}

\paragraph{\textbf{Star generating set}}
Our motivation for this next generating set comes from star transpositions in $\Symm_n$. Besides the set of Coxeter generators $\{(i\, i+1): 1\leq i\leq n\}$, the symmetric group $\Symm_n$ can also be generated by the set $S=\{ (1\, i): 2\leq i \leq n\}$. Observe that the corresponding labelled graph $(V, E)$ where the vertex set is  $V=\{1, \ldots, n\}$ and the edge set is $E=\{e_{ij}: (i\, j)\in S\}$ is star-shaped. Thus call elements in $S$ \emph{star transpositions}. Pak counted in \cite{Pak98} the number of reduced expressions of a permutation $\pi\in\Symm_n$ that fixes $1$ and has $m$ cycles of length $k\geq 2$. This result was generalized to any permutation in $\Symm_n$ by Irving and Rattan. They showed in \cite{IR09} that the number of minimal star factorizations of $\pi \in \Symm_n$ with cycles of lengths $\ell_1$, $\ldots$, $\ell_m$ including exactly $k$ fixed points not equal to $1$ is 

\begin{equation} \label{equ: star} 
\frac{(n+m-2(k+1))!}{(n-k)!} \ell_1 \cdots \ell_m.
\end{equation}

For small values of $n$, we enumerate fully commutative elements by length $\ell$ in $\Symm_n$ and list them in Table \ref{star in symmetric}. Observe that at $\ell=t>0$, the number of fully commutative elements in $\Symm_n$ appears to be $ (n-1)\cdot (n-2)\cdots (n-t)$. 

\begin{table}[!htbp]
\centering
\scalebox{1.0}{
\begin{tabular}{|ccccc|}
\hline$\ell$ & $n=3$ & $n=4$ & $n=5$ & $n=6$ \\
\hline 0 & 1 & 1 & 1 & 1 \\
\hline 1 & 2 & 3 & 4 & 5 \\
\hline 2 & 2 & 6 & 12 & 20 \\
\hline 3 & & 6 & 24 & 60 \\
\hline 4 & & & 24 & 120 \\
\hline 5 & & & & 120 \\
\hline
\hline total &5 &16 &65 &326  \\
\hline
\end{tabular}}
\vspace{0.1cm}
\caption{f.c. elements in $\Symm_n$ with star generating set}\label{star in symmetric}
\end{table}

\begin{proposition}
Consider the symmetric group $\Symm_n$ with the generating set $\{ (1\, i): 2\leq i \leq n\}$. The number of fully commutative elements is 
$1+\sum\limits_{t=1}^{n-1} \prod\limits_{j=1}^{t} (n-j).$ Furthermore, at length $\ell=t>0$, fully commutative elements are the $(t+1)$-cycles that move $1$. 
\end{proposition}

\begin{proof}
Let $\pi \in \Symm_n$ be a permutation with cycles of cycle lengths $\ell_1$, $\ldots$, $\ell_m$ including exactly $k$ fixed points not equal to $1$. Assume that $\pi$ is fully commutative. Since no two star transpositions commute with each other, full commutativity means unique reduced expression. Thus, we are looking for $\pi\in \Symm_n$ such that \eqref{equ: star} should give $1$. 

When $\pi=\id$, it is fully commutative trivially. In this case, $k=n-1$, $m=n$, $\ell_1=\ldots=\ell_m=1$ and $\frac{(n+m-2(k+1))!}{(n-k)!}=1$. When $\pi\neq\id$, we have $\ell_1\cdots\ell_m >1$. In order for \eqref{equ: star} to be $1$, then $\frac{(n+m-2(k+1))!}{(n-k)!}=\frac{(n-k + (m-k-2))!}{(n-k)!}<1$. This means $m-k-2<0$, i.e. $k>m-2$. Since $k$ is the number of fixed points not equal to $1$ and $m$ is the number of cycles in $\pi$, then $k<m$. This forces $k=m-1$. Since $\pi$ is fully commutative and has $m$ cycles including $k=m-1$ fixed points not equal to $1$, then only one cycle in $\pi$ has cycle length bigger than $1$. Since $\frac{(n+m-2(k+1))!}{(n-k)!}=\frac{1}{n-m+1}$, the cycle length of that one cycle is $n-m+1$. Since $\pi\neq\id$, this cycle moves $1$. This implies that $\pi$ is a cycle of the form $(1\, i_1\, i_2\, \ldots\, i_{n-m})$ in $\Symm_n$ where $k=m-1$ and $\ell_1\cdots\ell_m=n-m+1$. Plugging them into \eqref{equ: star} gives $1$. Since there are $(n-1)(n-2)\cdots (n-(n-m))$ of them for every value of $m$ in $\{ 1, \ldots, n-1\}$, the result follows. 
\end{proof}

Now let $s_i = [(1\, \,i+1); (0,\ldots, 0)]$ and $s_{\bar{1}} = [(1\, 2); (-1,1,0,\ldots, 0)]$. The group $G(m,m,n)$ can  be generated by elements in the set $S^* = \{s_{\bar{1}}, s_1, s_2, \ldots, s_{n-1}\}$. Denote $S^*$ as the \emph{star generating set} of $G(m,m,n)$.

When $n=3$, the star generating set $S^*$ has the exact same presentation as the classical generating set $S^c$. It follows that the enumeration of fully commutative elements by length is the same for both sets. But the actual fully commutative elements are different, as the generators in $S^*$ and in $S^c$ are not identical. Notably, element $[(13);(0,0,0)]$ is not fully commutative in $G(m,m,3)$ with $S^c$, because it has the pattern $321$ that fully commutative elements avoid. But in $G(m,m,3)$ with $S^*$, this element is fully commutative, since it is one of the generators. 

In general, since no two generators in the star generating set are commutative, full commutativity means unique reduced expression. In Table \ref{star}, we list the number of fully commutative elements in some $G(m,m,4)$ by length $\ell$. 

\begin{table}[!htbp]
\centering
\scalebox{0.7}{
\begin{tabular}{|cccccc|}
\hline
$\ell$ & $m=2$ & $m=3$ & $m=4$ & $m=5$ & $m=6$ \\
\hline 0 & 1 & 1 & 1 & 1 & 1 \\
\hline 1 & 4 & 4 & 4 & 4 & 4 \\
\hline 2 & 11 & 12 & 12 & 12 & 12 \\
\hline 3 & 20 & 24 & 26 & 26 & 26 \\
\hline 4 & 20 & 36 & 44 & 46 & 46 \\
\hline 5 & 8 & 44 & 68 & 76 & 78 \\
\hline 6 & 8 & 48 & 92 & 116 & 124 \\
\hline 7 & & 20 & 96 & 152 & 176 \\
\hline 8 & & & 68 & 176 & 232 \\
\hline 9 & & & 28 & 124 & 232 \\
\hline 10 & & & & 60 & 220 \\
\hline 11 & & & & 24 & 128 \\
\hline 12 & & & & & 40 \\
\hline 13 & & & & &20 \\
\hline
\hline total &72 &189 &439 & 817&1339 \\
\hline
\end{tabular}
}
\vspace{0.1cm}
\caption{f.c. elements in $G(m,m,4)$ with star generating set}
\label{star}
\end{table}

\subsection{Shephard groups}
A \emph{unitary reflection} is a linear transformation (of finite oder) of a complex vector space that fixes a hyperplane. A finite group generated by unitary reflections are \emph{unitary reflection groups}. Both finite real and complex reflection groups are unitary reflection groups. 

For a detailed treatment on Shephard groups, see \cite{Cox91}. \emph{Shephard groups} are unitary reflection groups that are the symmetry groups of \emph{regular complex polytopes} defined and classified by Shephard \cite{She52}. A Shephard group $G$ with associated positive integers $p_1$, \ldots, $p_n$ and $q_1$, \ldots, $q_{n-1}$ has the following presentation with respect to the generating set $\{s_1, \ldots, s_n\}$ \cite[13.4]{Cox91}:

\[\begin{aligned}
 s_i^{p_i} &= 1, &\text{ for } i = 1, \ldots, n, \\ 
 s_i s_j &= s_j s_i &\text{ if }|i-j|>1, \\
\underset{q_i\text{ letters}}{\underbrace{s_i s_{i+1} s_i s_{i+1} \cdots}}
 &  = 
\underset{q_i\text{ letters}}{\underbrace{ s_{i+1} s_i s_{i+1} s_i \cdots}} &\text{ for } i = 1, \ldots, n-1.
\end{aligned}\]

We use the symbol $p_1[q_1]p_2[q_2]\cdots[q_{n-1}]p_n$ to denote the Shephard group $G$ with the above presentation. 


Compared to the classification of complex reflection groups, the classification of Shephard groups is relatively short. There is one infinite family $m[4]2[3]2[3]\cdots 2[3]2$, which is $G(m,1,n)$ and a finite list of exceptional Shephard groups:
\begin{enumerate}
\item (real reflection groups) $G_{23}=H_3$, $G_{28}=F_4$, $G_{30}=H_4$,
\item (of rank $2$) $G_4$, $G_5$, $G_6$, $G_8$, $G_9$, $G_{10}$, $G_{14}$, $G_{16}$, $G_{17}$, $G_{18}$, $G_{20}$, $G_{21}$,
\item $G_{25}=3[3]3[3]3$,
\item $G_{26}=2[4]3[3]3$, and
\item $G_{32}=3[3]3[3]3[3]3$.
\end{enumerate}

In \cite{S98}, Stembridge enumerated fully commutative elements in Coxeter groups $A_n$, $B_n$, $D_n$, $E_n$, $F_n$ and $H_n$. In particular, he showed that the number of fully commutative elements in $H_n$ \cite[(3.4)]{S98} is \[{2n+2 \choose n+1} - 2^{n+2} +n +3 \]and in $F_n$ \cite[(3.7)]{S98} is \[5 f_{3n-4} - 5\sum\limits_{k=2}^{n-1}\frac{f_{3k-5}}{n-k+1}{2n-2k \choose n-k}+\frac{1}{n}{2n-2 \choose n-1} -2f_{2n-2}-2f_{2n-4}+f_{n-1}-1\] where $f_n$ is the $n$-th Fibonacci number. Thus $G_{23}=H_3$ has $44$ fully commutative elements, $G_{28}=F_4$ has $106$ fully commutative elements and $G_{30}=H_4$ has $195$ fully commutative elements.

In Table \ref{rank 2}, we enumerate fully commutative elements in exceptional Shephard groups by length $\ell$ and leave readers with the following conjecture and a few questions. 

\begin{conjecture}
Let $a,b$ be positive integers. Consider groups $a[b]a$ and $b[a]b$. The enumeration by length of fully commutative elements with unique reduced expression is identical in both groups. 
\end{conjecture}

This is true when $a=3$ and $b=4$, i.e. $G_5=3[4]3$ and $G_8 = 4[3]4$. This is also true when $a=5$ and $b=3$, i.e., $G_{16}=5[3]5$ and $G_{20}=3[5]3$ . Further, this is true when $a=2$ and $b$ is some positive integer. Observe that the group $2[m]2=G(m,m,2)$ has $2m-1$ fully commutative elements by Proposition \ref{prop: mm2}. At every length $\ell$ $(\ell >0)$, there are two fully commutative elements. And the group $m[2]m$ (product of two cyclic groups of order $m$) has $m^2$ fully commutative elements, i.e., every element in $m[2]m$ is fully commutative. But $2m-1$ of them have unique reduced expression and at every length $\ell$ $(\ell>0)$, there are two of them (one from each cyclic group), which is the exact behavior of fully commutative elements in $2[m]2$.

\begin{question}
Why do all exceptional Shephard groups except $G_{23}$ and $G_{28}$ have an odd number of fully commutative elements? 
\end{question}

\begin{question}
Let $G$ be an exceptional Shephard group. The sequence of fully commutative elements in $G$ ordered by increasing length follows an increasing-then-decreasing pattern, except when $G=G_{26}$. Why is $G_{26}$ the only exception?
\end{question}

\begin{table}[!htbp]
\centering
\scalebox{0.7}{
\begin{tabular}{|c|cccccccccccc|ccc|ccc|}
\hline
$\ell$ & $G_4$ & $G_5$ & $G_6$ & $G_8$ & $G_9$& $G_{10}$& $G_{14}$ & $G_{16}$ & $G_{17}$ & $G_{18}$&$G_{20}$&$G_{21}$&$G_{23}=H_3$ & $G_{25}$& $G_{26}$&$G_{28}=F_4$&$G_{30}=H_4$&$G_{32}$\\
\hline 0&1&1&1&1&1&1&1&1&1&1&1&1&1&1&1&1&1&1\\
\hline 1&2&2&2&2&2&2&2&2&2&2&2&2&3&3 &3&4&4&4\\
\hline 2&4&4&3&4&3&4&3&4&3&4&4&3&5& 8&7&9&9&13\\
\hline 3&4&6&4&6&5&7&4&6&5&7&6&4 &6&14&13&14&14&30\\
\hline 4&4&8&5&8&7&10&5&10&8&11&10&5 &7&22&20&18&18&61\\
\hline 5&&10&7&10&10&15&7&14&12&17&14&7&7&24&25&18 &21&100\\
\hline 6&&10&7&10&13&20&9&20&16&25&20&9&5&25 &28&16&23&152\\
\hline 7&&2&6&2&16&25&12&26&22&35&26&12&4&20&30 &12&21&204\\
\hline 8&&&2&&18&23&14&28&29&45&28&16&3&18&26&8& 20&270\\
\hline 9&&&&&17&18&15&20&36&52&20&21&2&14 &21&4&18&346\\
\hline 10&&&&&14&14&14&18&42&58&18&26 &1&12&21&2&16&450\\
\hline 11&&&&&6&5&11&12&49&61&12&31&&&23 &&12&556\\
\hline 12&&&&&2&1&6&8&51&65&8&35&&& 21&&8&686\\
\hline 13&&&&&1&&2&2&49&58&2& 39&&&19&&4&834\\
\hline 14&&&&&&&&&43&43&&38&&&11 &&3&1020\\
\hline 15&&&&&&&&&38&35&&38&&&2 &&2&1206\\
\hline 16&&&&&&&&&33&29&&33&&&&&1&1384 \\
\hline 17&&&&&&&&&28&22&&31 &&&&&&1494\\
\hline 18&&&&&&&&&16&15&&26&&&&&&1544 \\
\hline 19&&&&&&&&&7&5&&18&&&&& &1366\\
\hline 20&&&&&&&&&2&1&&13&&&&&& 1016\\
\hline 21&&&&&&&&&1&&&11&&&&&&620 \\
\hline 22&&&&&&&&&&&&6&& &&&&298\\
\hline 23&&&&&&&&&&&&2&&&&&&86 \\
\hline
\hline total &15&43&37&43&115&145&105&171&493&591&171&427& 44&161&271&106&195&13741\\
\hline
\end{tabular}
}
\vspace{0.1cm}
\caption{f.c. elements in exceptional Shephard groups}
\label{rank 2}
\end{table}

\end{document}